\documentclass[12pt, a4paper]{article}
\usepackage{amsthm,amsfonts,amsmath,amssymb,bookmark,appendix}

\theoremstyle{plain}
\newtheorem{theorem}{Theorem}
\newtheorem{lemma}[theorem]{Lemma}

\theoremstyle{definition}

\theoremstyle{remark}
\newtheorem*{remark}{Remark}

\numberwithin{equation}{section}

\DeclareMathOperator{\supp}{supp}

\DeclareMathOperator{\divg}{div}
\DeclareMathOperator{\rot}{rot}

\DeclareMathOperator{\loc}{loc}
\DeclareMathOperator{\dist}{dist}

\renewcommand{\phi}{\varphi}
\newcommand{\norm}[1]{\lVert #1 \rVert}

\newcommand{\eps}{\varepsilon}
\newcommand{\R}{\mathbb{R}}

\title{\textbf{A Magnetic Contribution to the Hardy Inequality}}
\author{Tomas Ekholm, Fabian Portmann}

\begin{document}
\maketitle

\begin{abstract}
	We study the quadratic form associated to the kinetic energy operator in the presence of an external magnetic field in $d = 3$. We show that if the radial 
	component of the magnetic field does not vanish identically, then the classical lower bound given by Hardy is improved by a non-negative potential term depending on 
	properties of the magnetic field.
\end{abstract}

\section{Overview and Results}
\subsection{Introduction}
The classical Hardy inequality states that for $d\geq 3$,
\begin{align}\label{eq:hardy_ineq}
	\int_{\R^d} \left|\nabla u(x)\right|^2\,dx - \frac{(d-2)^2}{4}\int_{\R^d}\frac{|u(x)|^2}{|x|^2}\,dx \geq 0
\end{align}
holds for any function $u \in C_c^{\infty}(\R^d)$, where the constant $(d-2)^2/4$ is sharp. There has been a lot of work surrounding the above inequality, for a review see for example \cite{D}. If we define the quadratic form
\begin{align*}
	h[u] := \int_{\R^d} \left|\nabla u(x)\right|^2\,dx - \frac{(d-2)^2}{4}\int_{\R^d}\frac{|u(x)|^2}{|x|^2}\,dx
\end{align*}
on $C_c^{\infty}(\R^d)$, \eqref{eq:hardy_ineq} tells us that $h$ is non-negative. Through the Friedrichs extension we can thus define the operator
\begin{align*}
	H := -\Delta-\frac{(d-2)^2}{4|x|^2}.
\end{align*}
The operator $H$ is then again non-negative and in fact also critical. This means that subtracting any non-negative potential $w$ from $H$ will immediately create negative spectrum, see for example \cite{W1} and \cite{EF}.

We stress that the assumption $d\geq3$ in \eqref{eq:hardy_ineq} and the above discussion is crucial; in $d=2$ the classical Hardy inequality is trivial, since no estimate of the form
\begin{align}\label{eq:hardy_ineq_2d}
	\int_{\R^2} \left|\nabla u(x)\right|^2\,dx \geq \int_{\R^2}w(x)|u(x)|^2\,dx
\end{align}
with a non-trivial $w \geq 0$ holds for all $u \in C_c^{\infty}(\R^2)$. The failure of any inequality of the type \eqref{eq:hardy_ineq_2d} shows that the operator $-\Delta$ by itself is already critical. The corresponding virtual level for the two-dimensional Laplace operator was investigated in \cite{BS} and \cite{S1}.

When studying physical systems it is relevant to estimate the kinetic energy in the presence of an external magnetic field $B: \R^3 \to \R^3$, satisfying $\divg B=0$. In this case the momentum operator $(-i\nabla)$ is replaced with the magnetic operator $(-i\nabla - A)$, $A: \R^3 \to \R^3$ being a magnetic vector potential with $\rot A = B$. The kinetic energy is now linked to the operator $(-i\nabla -A)^2$, defined via the magnetic Dirichlet form
\begin{align}\label{eq:magn_dir_form}
	\int_{\R^3}|(-i\nabla - A(x))u(x)|^2 \,dx
\end{align}
with form domain
\begin{align*}
	H_A^1(\R^3) = \{u \in L^2(\R^3): (-i\partial_j - A_j)u \in L^2(\R^3), j=1,2,3\}.
\end{align*}
We will henceforth assume that $A_j \in L_{\loc}^2(\R^3)$ for $j=1,2,3$. In \cite{S2} it was shown that for this class of magnetic vector potentials, $C_c^{\infty}(\R^3)$ is a form core for \eqref{eq:magn_dir_form}, allowing us to study $(-i\nabla-A)^2$ via \eqref{eq:magn_dir_form} on smooth functions instead. Furthermore, the pointwise diamagnetic inequality holds (see for example \cite{LL})
\begin{align*}
	|(-i\nabla - A(x))u(x)| \geq |\nabla|u(x)||,
\end{align*}
showing that if $u \in H_A^1(\R^3)$, then $|u| \in H^1(\R^3)$. Hence, one can recover Hardy's estimate even in the magnetic case; any $u \in C_c^{\infty}(\R^3)$ satisfies
\begin{align}\label{eq:magn_hardy_ineq}
	\int_{\R^3}|(-i\nabla - A(x))u(x)|^2 \,dx &\geq \int_{\R^3}|\nabla|u(x)||^2 \,dx\nonumber\\
	&\geq \frac{1}{4}\int_{\R^3}\frac{|u(x)|^2}{|x|^2}\,dx.
\end{align}
Since a magnetic field a priori only improves the nature of (or "lifts") the bottom of the spectrum, one is tempted to ask whether the operator
\begin{align*}
	H_A := (-i\nabla - A)^2-\frac{1}{4|x|^2},
\end{align*}
again defined through the Friedrichs extension of its corresponding quadratic form and known to be non-negative from \eqref{eq:magn_hardy_ineq}, is still critical. In \cite{W2}, the author showed that the introduction of any non-trivial magnetic field will always remove the criticality. For any compact $\Omega \subset \R^3$ there exists an inexplicit constant $C=C(\Omega,B)$ such that
\begin{align*}
	\int_{\R^3}|(-i\nabla - A(x))u(x)|^2 \,dx - \frac{1}{4}\int_{\R^3}\frac{|u(x)|^2}{|x|^2}\,dx \geq C \int_{\Omega}|u(x)|^2\,dx,
\end{align*}
hence it is possible to add a non-trivial potential term to the right hand side of \eqref{eq:magn_hardy_ineq}. The method however allows no control of the constant $C(\Omega,B)$.

Furthermore, in \cite{LW} it was shown that in two dimensions the presence of certain types of magnetic field allows for explicit Hardy inequalities. For these magnetic vector fields $B$, we have
\begin{align*}
	\int_{\R^2}|(-i\nabla - A(x))u(x)|^2 \,dx \geq C \int_{\R^2}\frac{|u(x)|^2}{1+|x|^2}\,dx,
\end{align*}
for any $u \in C_c^{\infty}(\R^2)$, where the explicit constant $C$ strongly depends on the behaviour of the flux $\Phi(r) = \frac{1}{2\pi}\int_{|x|<r} B\,dx$. In the special case of a Ahronov-Bohm field
\begin{align*}
	A(x) = \alpha\left(\frac{-x_2}{x_1^2+x_2^2}, \frac{x_1}{x_1^2+x_2^2}\right),
\end{align*}
where $\alpha \in \R$ is the total flux, one obtains
\begin{align*}
	\int_{\R^2}|(-i\nabla - A(x))u(x)|^2 \,dx \geq \min_{k \in \mathbb{Z}}(k - \alpha)^2 \int_{\R^2} \frac{|u(x)|^2}{|x|^2}\,dx
\end{align*}
for any $u \in C_c^{\infty}(\R^2\setminus\{0\})$.

Inspired by the above, we show that in the presence of a magnetic field with non-trivial radial component, \eqref{eq:magn_hardy_ineq} can be improved by adding an explicit non-negative potential term $\int_{\R^3}w(x)|u(x)|^2\,dx$ to the righthand side.
\begin{remark}
	Note that $w$ should depend on $B$ and not on $A$ due to the gauge invariance of the problem. When considering a general gauge transformation
	\begin{align*}
		A(x) \mapsto A(x) + \nabla \Psi(x), \quad \Psi \in C^1(\R^3),
	\end{align*}
	we see that $H_A$ and $H_{A + \nabla \Psi}$ are unitarily equivalent by the unitary transformation $U(x) = e^{i\Psi(x)}$, hence the spectrum remains unchanged and 
	$w$ should depend on the physical quantity $B$ only. In the case of a pure gauge field $A = \nabla \Psi$, $w$ is identically zero, illustrating the fact that we are 
	dealing with the original critical operator $H$.
\end{remark}

\subsection{Main Results}
We will, unless explicitly stated, always assume that we are working with a magnetic field $B$ fulfilling the reasonable physical conditions.
\paragraph{Assumptions on the magnetic field $B$:}
\begin{align}\label{assumption:A1}
	&B:\R^3 \to \R^3,\nonumber\\ 
	&B_j \in C^1(\R^3), \quad j=1,2,3,\\
	&\divg B = 0\nonumber.
\end{align}

For a given $B$ we denote its radial component by $B_r$. Now let $A$ be any magnetic vector potential with $\rot A = B$ and let $h_A$ be the quadratic form
\begin{align}\label{eq:quad_a_def}
	h_A[u] := \int_{\R^3}|(-i\nabla - A(x))u(x)|^2 \,dx - \frac{1}{4}\int_{\R^3}\frac{|u(x)|^2}{|x|^2}\,dx,
\end{align}
initially defined on the space $C_c^{\infty}(\R^3)$, giving rise to the operator $H_A$ via the Friedrichs extension. Our first result is
\begin{theorem}\label{thm:main_magn_est_1}
	Assume that there exists an $\omega_0 \in \mathbb{S}^2$ and $R>0$ such that  $B_r(R\omega_0)\neq 0$. Then 
	\begin{align*}
		h_A[u] \geq C_1 \int_{\R^3}\frac{|u(x)|^2}{1+|x|^2\log^2\frac{|x|}{R}}\,dx
	\end{align*}
	holds for any $u \in C_c^{\infty}(\R^3)$, with the constant $C_1$ only depending on $B$.
\end{theorem}
\begin{remark}
	A similar result can be obtained in $d=2$, something that was already observed in \cite{K} with an inexplicit constant.
\end{remark}

Denote by
\begin{align*}
	\Phi(r,\theta)= \frac {1}{2\pi}\int_{S(r,\theta)}B_r\,dS
\end{align*}
the flux through the capped sphere $S(r,\theta)$ of opening angle $\theta$ and radius $r$. It is then possible to remove the logarithmic term if the flux does not vanish at infinity.
\begin{theorem}\label{thm:main_magn_est_2}
	If there exist constants $\theta_0, \theta_1 \in (0,\pi)$ and $C_2,C_3 \in (0,1)$, all independent of $r,\theta$ such that $\Phi(r,\theta)$ satisfies the estimate
	\begin{align*}
		C_2 \leq \Phi(r,\theta) \leq C_3,
	\end{align*}
	for $\theta_0 \leq \theta \leq \theta_1$ and $r> r_0$, then 
	\begin{align*}
		h_A[u] \geq C_4\int_{\R^3}\frac{|u(x)|^2}{1+|x|^2}\,dx
	\end{align*}
	holds for any $u \in C_c^{\infty}(\R^3)$, where $C_4$ is a constant depending on $B$.
\end{theorem}
\begin{remark}
	Theorem~\ref{thm:main_magn_est_1} will typically apply to fields with compact support, whereas for Theorem~\ref{thm:main_magn_est_2} one would in some sense need a behaviour of $B_r \sim r^{-2}$ at infinity.
\end{remark}

\subsection{Comparison to Existing Results}
We would like to compare our result to some existing bounds. We begin with the classical commutator estimate proved in \cite{AHS},
\begin{align}\label{eq:comm_est_ineq}
	(u,(-i\nabla - A)^2u) \geq |(B_{jk}u,u)|,
\end{align}
where $B_{jk} = \partial_j A_k - \partial_k A_j$ with $j,k=1,2,3$. A variant of the above is the result given in \cite{ET}, where the authors obtain a bound for $h_A$, yet their potential is not necessarily non-negative. The major problem here is that in the general case the quantities $B_{jk}$ can change sign, leading to possible cancellations and hence the positivity of the kinetic energy operator remains veiled. The result of \cite{BLS} remedies the above in a certain sense, namely the authors obtain a bound of the form
\begin{align*}
	\int_{\R^3}|(-i\nabla - A(x))u(x)|^2 \,dx \geq C \int_{\R^3}|B(x)||u(x)|^2\,dx,
\end{align*}
under some regularity assumptions on $B$, assuming that $B$ never vanishes and does not change direction too quickly. Our result has the advantage that it incorporates \eqref{eq:magn_hardy_ineq} and reduces to the standard Hardy inequality in the case of $B = 0$. Its application is best suited to situations where the essential spectrum of $H_A$ is $[0,\infty)$ and there is no gap in the spectrum. This is for example the case if one considers magnetic fields $B$ which go to zero at infinity, see \cite{L}. 

Our method does not give an extra contribution to $h_A$ for magnetic fields that live on the sphere, i.e. fields for which
\begin{align*}
	B_r \equiv 0.
\end{align*}
For these fields the flux through any spherical cap is always zero, hence $w$ vanishes identically. This is a flaw stemming purely from our method, since we know from \cite{W2} that an improvement is possible for any non trivial $B$. Note that these fields on the sphere have to vanish at some point (due to the continuity assumption), since one "can not comb the hairs on a sphere". If we assume however that the field never vanishes, it is the same theorem that tells us that $B$ has to point in the radial direction somewhere if it is to be continuous and we will obtain an extra contribution.
\begin{remark}
	A solution is to choose a different origin $y$ for the spherical coordinates so that the radial component no longer vanishes identically. 
\end{remark}

\section{Magnetic Estimates}
\subsection{Preliminaries}
If $B$ satisfies \eqref{assumption:A1}, then it is always possible to construct a sufficiently regular magnetic vector potential $A$ with
\begin{align}\label{eq:multipolar_gauge}
	x \cdot A(x) = 0
\end{align} 
and $\rot A = B$ through a method stated in \cite{Y},
\begin{align*}
	A(x) = \int_0^1 (B(xs) \times xs)\,ds.
\end{align*}
The reason for choosing this gauge is that in polar coordinates the radial part alone already compensates for the classical Hardy term. When changing to spherical coordinates $x = (r,\omega)$, the quadratic form $h_A$ turns into
\begin{align*}
	h_A[u] = \int_0^\infty \int_{\mathbb{S}^2} \left[|\partial_r u|^2 + \frac{1}{r^2}\left|(-i\nabla_{\omega}-rA_{\omega})u\right|^2 - \frac{|u|^2}{4r^2}\right]\, r^2d\omega dr,
\end{align*}
$d\omega$ being the surface measure on $\mathbb{S}^2$.

For $u \in C_c^{\infty}(\R^3)$, let
\begin{align*}
	l[u] := \int_0^\infty \int_{\mathbb{S}^2} |\partial_r (\sqrt{r}u)|^2\, rd\omega dr.
\end{align*}
With a ground-state substitution we obtain 
\begin{align*}
	h_A[u] = l[u] + \int_0^\infty g_A[u]\,dr,
\end{align*}
where $g_A[u]$, for fixed $r>0$, is given by
\begin{align*}
	g_A[u] := \int_{\mathbb{S}^2} \left|(-i\nabla_{\omega}-rA_{\omega})u\right|^2\, d\omega.
\end{align*}
It is now clear that the radial part of the Schr\"odinger operator $H_A$ behaves like the radial part of the two-dimensional Laplace operator, explaining the appearance of the logarithmic term in Theorem~\ref{thm:main_magn_est_1}.

Next, we introduce the standard spherical coordinates on $\mathbb{S}^2$
\begin{align*}
	x_1 &= r \cos\phi \sin\theta\\
	x_2 &= r \sin\phi \sin \theta\\
	x_3 &= r \cos \theta,
\end{align*}
where $(r,\theta,\phi) \in [0,\infty)\times[0,\pi)\times[0,2\pi)$. In these coordinates,
\begin{align*}
	g_A[u] = \int_0^{2\pi}\int_0^{\pi}\left[\left|(-i\partial_{\theta} - rA_{\theta})u\right|^2 + \left|\left(\frac{-i \partial_\phi}{\sin \theta} - r A_{\phi}\right) u\right|^2\right]\,\sin\theta d
	\theta d\phi,
\end{align*}
where $A_{\theta} := A \cdot e_{\theta}$ and $A_{\phi}:=A \cdot e_{\phi}$. Here $e_{\theta}$ and $e_{\phi}$ are the standard unit vectors given the above spherical coordinates.

\begin{theorem}\label{thm:tech_magn_est_1}
	For $u \in C_c^{\infty}(\R^3)$, the inequality
	\begin{align*}
		 \int_0^{2\pi}\int_0^{\pi}\left|\left(\frac{-i \partial_\phi}{\sin \theta} - r A_{\phi}\right) u\right|^2\,\sin\theta d\theta d\phi \geq  \int_0^{2\pi}\int_0^{\pi} w_1(r,\theta) |u|^2\,\sin\theta d\theta d\phi
	\end{align*}
	holds, where $w_1(r,\theta)$ is the non-negative potential
	\begin{align*}
		w_1(r,\theta) := \frac{\min_{k \in \mathbb{Z}}\left(k-\Phi(r,\theta)\right)^2}{\sin^{2}\theta}.
	\end{align*}
\end{theorem}

For a given $r>0$, one readily observes that if the flux of the radial component of the magnetic field vanishes through one of the capped spheres with opening angle $\theta$, then $w_1(r,\theta)$ will also vanish at this point. Hence $w_1$ does not serve as an optimal candidate for a lower bound on the quadratic form $h_A$. This can be remedied with the following theorem, where we show that as long as the flux is positive through some of the spherical caps, then the potential is positive on the whole sphere.
\begin{theorem}\label{thm:tech_magn_est_2}
	Assume that $M(r) \leq w_1(r,\theta)$, for $0 < \theta_0(r) \leq \theta \leq \theta_1(r) <\pi$ and $r>0$, then the inequality
	\begin{align*}
		g_A[u]\geq \int_0^{2\pi}\int_0^{\pi}w_2(r)|u(r,\theta,\phi)|^2\,r^2\sin\theta d\theta d\phi
	\end{align*}
	holds for any $u \in C_c^{\infty}(\R^3)$, where $w_2$ is (in spherical coordinates) given by 
	\begin{align*}
		w_2(r) = \frac{\lambda(M(r),\theta_0(r),\theta_1(r))}{r^2}.
	\end{align*}
	For the definition of the function $\lambda$ see \eqref{def:lambda_const}.
\end{theorem}

If the radial component $B_r$ should vanish identically on a sphere, then $M(r) = 0$, and $\lambda(M(r),\theta_0(r),\theta_1(r))=0$ ($\theta_0(r)$, $\theta_1(r)$ can be arbitrary). It is then still possible to obtain a strictly positive radial potential if the potential is positive on some spherical shells. By the continuity of the magnetic field $B$ there might exist intervals $I_j = (\alpha_j,\beta_j)$ and constants $L_j > 0$ so that $w_2(r) \geq \sum_{j=0}^{N}L_j\chi_{I_j}(r)$, $N$ being finite or inifite. Let $m_j$ be the midpoints of these intervals, defined as
\begin{align*}
	m_j = \left\{ \begin{array}{ll}
			\frac{\beta_j+\alpha_j}{2} & j \geq 0\\
			0 & j=-1.
			\end{array} \right.
\end{align*}
\begin{theorem}\label{thm:tech_magn_est_3}
	If $N>0$, then
	\begin{align*}
		h_A[u] \geq D_1 \int_{\R^3}\frac{|u(x)|^2}{1+|x|^2\log^2\left(\frac{2|x|}{\alpha_j+\beta_j}\right)}\,dx,
	\end{align*}
	$D_1$ depending on $B$.
\end{theorem}

\begin{theorem}\label{thm:tech_magn_est_4}
	If there exists constants $D_2,D_3,D_4$, depending only on $B$, such that for all $0\leq j \leq N$
	\begin{align}\label{thm:main_magn_est_4_cond}
		\frac{1}{L_j(1+\dist(0,I_j)^2)} &\leq D_2\nonumber\\
		\frac{m_j-m_{j-1}}{1+\dist(0,I_j)^2} &\leq D_3\\
		\max\left\{\frac{m_j - m_{j-1}}{1+\dist(0,I_j)^2}, \frac{m_{j+1} - m_{j}}{1+\dist(0,I_{j+1})^2}\right\} &\leq D_4 L_j |I_j|^2\nonumber,
	\end{align}
	then $u \in C_c^{\infty}(\R^3)$ satisfies
	\begin{align*}
		h_A[u] \geq D_5 \int_{\R^3}\frac{|u(x)|^2}{1+|x|^2}\,dx.
	\end{align*}
	The constant $D_5$ depends on $B$ only.
\end{theorem}

\subsection{Proofs of Theorems~\ref{thm:tech_magn_est_1} - \ref{thm:tech_magn_est_4}}
\begin{proof}[Proof of Theorem~\ref{thm:tech_magn_est_1}]
	We define
	\begin{align*}
		k_A[u] &:= \int_0^{2\pi}\int_0^{\pi}\left|\left(\frac{-i \partial_\phi}{\sin \theta} - r A_{\phi}\right) u\right|^2\,\sin\theta d\theta d\phi\\
		&= \int_0^{\pi}\frac{1}{\sin^2\theta}\left[\int_0^{2\pi}\left|K_{r,\theta}u\right|^2\,d\phi\right]\sin\theta\, d\theta.
	\end{align*}
	Here we have introduced the operator 
	\begin{align*}
		K_{r,\theta}u =(-i\partial_\phi - r \sin\theta A_{\phi}) u,
	\end{align*}
	which we will study with periodic boundary conditions on $H^1(0,2\pi)$ for fixed values of $r,\theta$, following the arguments given in \cite{LW}. The 
	eigenvalues are given by
	\begin{align*}
		\lambda_k(r,\theta) = k - \frac {r\sin\theta}{2\pi} \int_0^{2\pi} A_{\phi}(r,\phi',\theta) \, d\phi', \quad k \in \mathbb{Z},
	\end{align*}
	with corresponding eigenfunctions
	\begin{align*}
		\Pi_k(r,\theta,\phi) = \frac{1}{\sqrt{2\pi}} e^{-i \sin \theta \left(r\int_0^\phi A_{\phi} \, d\phi' - \frac{\phi r}{2\pi}\int_0^{2\pi}A_{\phi}\, d\phi' - \frac{k\phi}{\sin \theta}\right)}.
	\end{align*}
	Since the $\{\Pi_k\}_{k \in \mathbb{Z}}$ constitute a complete orthonormal system in $L^2(0,2\pi)$, we can write any function $u(r,\theta,\phi) \in L^2(\R^3)$ as
	\begin{align*}
		u(r,\theta,\phi) = \sum_{k \in \mathbb{Z}} u_k(r,\theta) \Pi_k(r,\theta,\phi).
	\end{align*}
	Replacing this representation into $k_A[u]$ and using Parseval's identity, we obtain
	\begin{align*}
		k_A[u] &= \int_0^{\pi}\frac{1}{\sin^2\theta}\left[\int_0^{2\pi}\left|K_{r,\theta}u\right|^2\,d\phi\right]\sin\theta\, d\theta\\
		&=\int_0^{\pi}\frac{1}{\sin^2\theta}\sum_{k\in \mathbb{Z}}|\lambda_k(r,\theta)|^2|u_k(r,\theta)|^2\, \sin\theta d\theta\\
		&\geq \int_0^\pi \frac{\min_{k \in \mathbb{Z}}\lambda_k^2(r,\theta)}{\sin^2\theta}\int_0^{2\pi}|u(r,\theta,\phi)|^2\, d\phi\,\sin\theta d\theta.
	\end{align*}
	From this we conclude that
	\begin{align*}
		h_A[u] &\geq \int_0^\infty k_A[u]\,dr\\
		&\geq \int_0^\infty \int_0^{\pi}\int_0^{2\pi}\frac{\min_{k \in \mathbb{Z}}\lambda_k^2(r,\theta)}{r^2\sin^2\theta}\,|u(r,\theta,\phi)|^2\, r^2 \sin\theta d\theta d\phi dr.
	\end{align*}
	Next we begin investigating the $\lambda_k(r,\theta)$ and parametrize $S(r,\theta)$ by
	\begin{align*}
		S(r,\theta) = \{(r \cos\phi' \sin\theta', r\sin\phi' \sin \theta', r \cos \theta'): 0\leq\phi'<2\pi,0\leq\theta'<\theta \}.
	\end{align*}
	We claim that the second term of $\lambda_k(r,\theta)$ is actually equal to the flux of $B$ through $S(r,\theta)$. To see this, first note that the normal vector to this 
	surface is simply $e_r$. To compute the flux of the magnetic field $B = \rot A$ through this surface, we write
	\begin{align}\label{eq:ev_flux_comp}
		\Phi(r,\theta) &= \frac{1}{2\pi}\int_{S(r,\theta)}B_r\,dS\nonumber\\
		&= \frac{1}{2\pi}\int_0^{2\pi}\int_{0}^{\theta}\frac{1}{r\sin\theta'}\left(\partial_{\theta'}(A_{\phi}(r,\theta',\phi')\sin\theta'\right))\,r^2\sin\theta'd\theta' d\phi'\nonumber\\
		&\quad  -\frac{1}{2\pi}\int_0^{2\pi}\int_{0}^{\theta}\frac{1}{r\sin\theta'}\partial_{\phi'}A_{\theta}(r,\theta',\phi')\,r^2\sin\theta' d\theta' d\phi'\nonumber\\
		&= \frac{r\sin\theta}{2\pi}\int_0^{2\pi}A_{\phi}(r,\theta,\phi')\,d\phi',
	\end{align}
	where we used the fact that $A_{\theta}$ is continuous in $\phi'$. Replacing this into the definition of the eigenvalues $\lambda_k(r,\theta)$ we obtain
	\begin{align*}
		\lambda_k(r,\theta) &= k - \frac {r\sin\theta}{2\pi} \int_0^{2\pi}A_{\phi}(r,\theta,\phi')\,d\phi'\\
		&= k - \Phi(r,\theta).
	\end{align*}
	The proof is complete.
\end{proof}

\begin{proof}[Proof of Theorem~\ref{thm:tech_magn_est_2}]
	To treat the first term of $g_A[u]$ we use the pointwise diamagnetic inequality
	\begin{align*}
		\left| -i \partial_\theta u - r A_{\theta} u\right| \geq \left|\partial_\theta |u|\right|.
	\end{align*}
	The second term of $g_A[u]$ is bounded as in the proof of Theorem~\ref{thm:tech_magn_est_1}, so that the whole quadratic form is 
	bounded from below by
	\begin{align*}
		g_A[u] \geq \int_0^{2\pi} \left[\int_0^{\pi} \left[|\partial_\theta |u||^2 + w_1(r,\theta)|u|^2\right]\sin\theta d\theta
		\right] d\phi.
	\end{align*}
	The above can then be seen as a Neumann problem for the differential operator in $\theta$ on a weighted space with the potential $w_1$. Our goal is to estimate the 
	lowest eigenvalue by a non-negative function, depending only on $r$ and $\theta_0,\theta_1$.
	\begin{lemma}\label{lem:low_ev_est}
		For any function $v \in H^1((0,\pi);\sin\theta d\theta)$ and non-negative potential $V$ satisfying
		\begin{align*}
			M < V(\theta), \quad 0<\theta_0 \leq \theta \leq \theta_1<\pi,
		\end{align*}
		for some $M>0$, we have the estimate
		\begin{align*}
			\int_0^{\pi} \left[|\partial_\theta v|^2 + V|v|^2\right]\sin\theta d\theta \geq \lambda(M,\theta_0,\theta_1)\int_0^{\pi}|v|^2\,\sin\theta d\theta,
		\end{align*}
		where
		\begin{align}\label{def:lambda_const}
			\lambda(M,\theta_0,\theta_1) = \frac{M}{2 + 4k_1((\theta_0+\theta_1)/2)M + 4k_1((\theta_0+\theta_1)/2)k_2(\theta_0,\theta_1)}.
		\end{align}
		The constants $k_1$ and $k_2$ are defined in \eqref{def:k_1_const}, \eqref{def:k_2_const}.
	\end{lemma}
	\begin{proof}[Proof of Lemma~\ref{lem:low_ev_est}]
		According to the assumptions we can find an interval $I=(\theta_0,\theta_1)$ where $0<M \leq V(\theta)$, and we denote the midpoint of this interval by $c 
		= (\theta_0 + \theta_1)/2$. Associated to this interval we choose a non-negative piecewise linear function $\zeta(\theta)$ given by
		\begin{align*}
			\zeta(\theta) = \left\{ \begin{array}{ll}
 						\frac{2}{(\theta_1-\theta_0)}|\theta-c| & \textrm{if $\theta \in I$}\\
						1 & \textrm{if $\theta \notin I$}
 					     \end{array} \right. 
		\end{align*}
		Clearly,
		\begin{enumerate}
			\item$0 \leq \zeta(\theta) \leq 1$,
			\item The derivative $\zeta'$ is supported on $I$ and 
			\begin{align}\label{def:k_2_const}
				|\zeta'(\theta)| \leq \frac{2}{(\theta_1-\theta_0)}=: k_2(\theta_0,\theta_1).
			\end{align}			
		\end{enumerate}
		We then write
		\begin{align*}
			\frac{1}{2}\int_0^{\pi}|v|^2\,\sin\theta d\theta \leq \int_0^{\pi}|\zeta v|^2\,\sin\theta d\theta + \int_0^{\pi}|(1-\zeta)v|^2\,\sin\theta d\theta.
		\end{align*}
		Next we need a lemma, who's proof is deferred to the Appendix.
		\begin{lemma}\label{lem:poincare_type_ineq}
			Let $0<c<\pi$. For all $f \in H^1((0,\pi);\sin\theta d\theta)$ with $f(c) = 0$, it holds that
			\begin{align*}
				\int_0^\pi |f(\theta)|^2\,\sin\theta d\theta \leq k_1(c)\int_0^{\pi}|f'(\theta)|^2\,\sin\theta d\theta,
			\end{align*}
			where
			\begin{align}\label{def:k_1_const}
				k_1(c) = \frac{\max \left\{c^2,(c-\pi)^2\right\}}{2\sin c}.
			\end{align}
		\end{lemma}
		The function $\zeta v$ vanishes for $\theta = c$, so Lemma~\ref{lem:poincare_type_ineq} yields
		\begin{align*}
			\frac{1}{2}\int_0^{\pi}|v|^2\,\sin\theta d\theta &\leq k_1\int_0^{\pi}|\partial_{\theta}(\zeta v)|^2\,\sin\theta d\theta + \int_{\theta_0}^{\theta_1}|(1-\zeta)v|^2\,\sin
			\theta d\theta\\
			&\leq 2k_1\int_0^{\pi}|\zeta\partial_{\theta}v|^2\,\sin\theta d\theta + 2k_1\int_0^{\pi}|\zeta' v|^2\,\sin\theta d\theta\\
			&\quad+ \int_{\theta_0}^{\theta_1}|(1-\zeta)v|^2\,\sin\theta d\theta\\
			&\leq 2k_1\int_0^{\pi}|\partial_{\theta}v|^2\,\sin\theta d\theta + 2k_1k_2\int_{\theta_0}^{\theta_1}|v|^2\,\sin\theta d\theta\\
			&\quad+ \int_{\theta_0}^{\theta_1}|v|^2\,\sin\theta d\theta\\
			&\leq 2k_1\int_0^{\pi}|\partial_{\theta}v|^2\,\sin\theta d\theta + (1+2k_1k_2)\int_{\theta_0}^{\theta_1}|v|^2\,\sin\theta d\theta.
		\end{align*}
		For the last term we know that on the specified interval, the potential is always larger than $M$, hence
		\begin{align*}
			(1+2k_1k_2)\int_{\theta_0}^{\theta_1}|v|^2\,\sin\theta d\theta &\leq \frac{(1+2k_1k_2)}{M}\int_{\theta_0}^{\theta_1}V|v|^2\,\sin\theta d\theta\\
			&\leq \frac{(1+2k_1k_2)}{M}\int_{0}^{\pi}V|v|^2\,\sin\theta d\theta.
		\end{align*}
		Combining the above estimates gives
		\begin{align*}
			\frac{1}{2}\int_0^{\pi}|v|^2\,\sin\theta d\theta &\leq \left(2k_1 + \frac{(1+2k_1k_2)}{M}\right)\int_0^{\pi} \left[|\partial_\theta v|^2 + V|v|^2\right]\sin\theta 
			d\theta,
		\end{align*}
		and rearranging the constants completes the proof.
	\end{proof}
	We then apply Lemma~\ref{lem:low_ev_est} to the function $|u|$ and the potential $w_1(r,\theta)$ for fixed $r>0$, so that
	\begin{align*}
		\int_0^{2\pi} \left[\int_0^{\pi} \left[|\partial_\theta |u||^2 + w_1(r,\theta)|u|^2\right]\sin\theta d\theta\right] d\phi\\ 
		\geq \frac{\lambda(M(r),\theta_0(r),\theta_1(r))}{r^2}\int_0^{2\pi}\int_0^{\pi}|u|^2\,r^2\sin\theta d\theta d\phi.
	\end{align*}
\end{proof}

\begin{proof}[Proof of Theorem~\ref{thm:tech_magn_est_3}]
	We keep $l[u]$ and treat $g_A[u]$ just as in the proof of Theorem~\ref{thm:tech_magn_est_2}, and arrive at
	\begin{align*}
		h_A[u] \geq \int_0^\infty \int_{\mathbb{S}^2}\left[ |\partial_r (\sqrt{r}u)|^2 + w_2(r)|\sqrt{r}u|^2\right]\,rd\omega dr.
	\end{align*}
	For the convenience we substitute $v:= \sqrt{r}u$ and want to show that
	\begin{align*}
		\int_0^{\infty}\left[|\partial_rv|^2 + w_2(r)|v|^2\right]\,r dr \geq D_1\int_0^{\infty}\frac{|v|^2}{1+r^2\log^2\frac{r}{m_j}}\,rdr.
	\end{align*}
	To do so, we define
	\begin{align*}
		\eta(r) = \left\{ \begin{array}{ll}
	 				m_j^{-1}|r-m_j| & \textrm{if $\alpha_j\leq r \leq \beta_j$}\\
					1 & \textrm{else}
	 				\end{array} \right. 
	\end{align*}
	Then,
	\begin{align*}
		\frac{1}{2}\int_0^{\infty} \frac{|v|^2}{1+(r\log(r/m_j))^2}\,rdr &\leq \int_0^{\infty} \frac{|\eta v|^2}{1+(r\log(r/m_j))^2}\,rdr\\
		&\quad + \int_0^{\infty} \frac{|(1-\eta)v|^2}{1+(r\log(r/m_j))^2}\,rdr
		\\
		&\leq \int_0^{\infty} \frac{|\eta v|^2}{1+(r\log(r/m_j))^2}\,rdr + \int_{\alpha_j}^{\beta_j}|v|^2\,rdr.
	\end{align*}
	We begin by estimating the first term and write
	\begin{align*}
		\int_0^{\infty} \frac{|\eta v|^2}{1+(r\log(r/m_j))^2}\,rdr &= \int_0^{m_j} \frac{|\eta v|^2}{1+(r\log(r/m_j))^2}\,rdr\\
		&\quad + \int_{m_j}^{\infty} \frac{|\eta v|^2}{1+(r\log(r/m_j))^2}\,rdr\\
		&\leq \int_0^{m_j} |\eta v|^2\,rdr + \int_{m_j}^{\infty} \frac{|\eta v|^2}{1+(r\log(r/m_j))^2}\,rdr.
	\end{align*}
	For any $f \in H^1(a,b)$ with $f(b) = 0$ we have 
	\begin{align*}
		\int_a^b|f(t)|^2t\,dt \leq \frac{(b-a)^2}{2}\int_a^b|f'(t)|^2t\,dt, 
	\end{align*}	
	and thus
	\begin{align*}
		\int_0^{m_j} |\eta v|^2\,rdr \leq \frac{m_j^2}{2} \int_0^{m_j} |\partial_r(\eta v)|^2\,rdr.
	\end{align*}
	Furthermore,
	\begin{align*}
		 \int_{m_j}^{\infty} \frac{|\eta v|^2}{1+(r\log(r/m_j))^2}\,rdr &\leq \int_{m_j}^{\infty} \frac{|\eta v|^2}{(r\log(r/m_j))^2}\,rdr\\
		 &\leq 4\int_{m_j}^{\infty}|\partial_r(\eta v)|^2\,rdr,
	\end{align*}	
	which follows from an integration by parts argument. Hence,
	\begin{align*}
		\int_0^{\infty} \frac{|\eta v|^2}{1+(r\log(r/m_j))^2}\,rdr \leq n_1 \int_0^{\infty}|\partial_r(\eta v)|^2\,rdr,
	\end{align*}
	where $n_1(m_j)= \max\left\{4,m_j^2/2\right\}$. From the properties of $\eta$ we deduce that 
	\begin{align*}
		n_1 \int_0^{\infty}|\partial_r(\eta v)\,rdr \leq 2n_1\int_0^{\infty}|\partial_rv|^2\,rdr + 2n_1n_2 \int_{\alpha_j}^{\beta_j}|v|^2\,rdr,
	\end{align*}
	with $n_2(\alpha_j,\beta_j) = \frac{2}{(\beta_j-\alpha_j)}$. Hence
	\begin{align*}
		\frac{1}{2}\int_0^{\infty} \frac{|v|^2}{1+(r\log(r/c))^2}\,rdr &\leq 2n_1\int_0^{\infty}|\partial_rv|^2\,rdr\\
		&\quad+ (2n_1n_2+1) \int_{\alpha_j}^{\beta_j}|v|^2\,rdr\\
		&\leq 2n_1\int_0^{\infty}|\partial_rv|^2\,rdr\\
		&\quad+ \frac{(2n_1n_2+1)}{L_j} \int_{\alpha_j}^{\beta_j}w_2(r)|v|^2\,rdr.
	\end{align*}
	Rearranging the constants gives
	\begin{align*}
		D_1 = \frac{L_j}{4n_1L_j + 4n_1n_2+2}.
	\end{align*}
\end{proof}

\begin{proof}[Proof of Theorem~\ref{thm:tech_magn_est_4}]
	As previously,
	\begin{align*}
		h_A[u] \geq \int_0^\infty \int_{\mathbb{S}^2}\left[ |\partial_rv|^2 + w_2(r)|v|^2\right]\,rd\omega dr,
	\end{align*}
	where $v= \sqrt{r}u$. We begin by defining the function $\xi(r) = \sum_{j=0}^{\infty}\xi_j(r)$, where
	\begin{align*}
		\xi_0 = \left\{ \begin{array}{ll}
 				1 & \textrm{if $0 \leq r < \alpha_0$}\\
				1-\frac{r}{m_0} & \textrm{if $\alpha_0 \leq r \leq m_0$}
 				\end{array} \right.
	\end{align*}
	and for $j \geq 1$,
	\begin{align*}
		\xi_j(r) = \left\{ \begin{array}{ll}
				\frac{r}{m_j} & \textrm{if $m_{j-1} \leq r < \beta_{j-1}$}\\
 				1 & \textrm{if $\beta_{j-1} \leq r < \alpha_j$}\\
				1-\frac{r}{m_0} & \textrm{if $\alpha_j \leq r \leq m_j$}.
 				\end{array} \right.
	\end{align*}
	We then have
	\begin{align*}
		\frac{1}{2}\int_0^{\infty}\frac{|v|^2}{1+r^2}\,rdr \leq \int_0^{\infty} \frac{|\xi v|^2}{1+r^2}\,rdr + \int_0^{\infty} \frac{|(1-\xi)v|^2}{1+r^2}\,rdr.
	\end{align*}
	Since the supports of the $\xi_j$ are disjoint, we obtain
	\begin{align*}
		\int_0^{\infty} \frac{|(1-\xi)v|^2}{1+r^2}\,rdr &= \sum_{j=1}^{\infty}\int_{I_j} \frac{|(1-\xi)v|^2}{1+r^2}\,rdr\\
		&\leq \sum_{j=1}^{\infty} \frac{1}{1+\alpha_j^2}\int_{I_j}|v|^2\,rdr\\
		&\leq \sum_{j=1}^{\infty} \frac{1}{L_j(1+\alpha_j^2)}\int_{I_j}w_2(r)|v|^2\,rdr\\
		&\leq D_2 \int_0^{\infty}w_2(r)|v|^2\,rdr.
	\end{align*}
	Next, set $K_j = \supp \xi_j$. Then
	\begin{align*}
		\int_0^{\infty} \frac{|\xi v|^2}{1+r^2}\,rdr &\leq \sum_{j=0}^{\infty}\frac{1}{1+\alpha_j^2}\int_{K_j}|\xi_jv|^2\,rdr\\
		&\leq \sum_{j=0}^{\infty}\frac{|K_j|^2}{1+\alpha_j^2}\left(\int_{K_j}|\xi_j'v|^2\,rdr + \int_{K_j}|\partial_rv|^2\,rdr\right)\\
		&\leq D_3 \int_0^{\infty}|\partial_r v|^2\,rdr + \sum_{j=0}^{\infty}\frac{|K_j|^2}{1+\alpha_j^2}\int_{K_j}|\xi_j'v|^2\,rdr.
	\end{align*}
	Now we study the term containing $\xi_j' v$ in more detail an obtain
	\begin{align*}
		\sum_{j=0}^{\infty}\frac{|K_j|^2}{1+\alpha_j^2}\int_{K_j}|\xi_j'v|^2\,rdr &\leq \max\left\{\frac{m_0^2}{1+\alpha_0^2},\frac{(m_1-m_0)^2}{1+
		\alpha_1^2}\right\}\frac{4}{|I_0|^2}\int_{I_0}|v|^2\,rdr\\
		&\quad+ \sum_{j=1}^{\infty}\max\left\{\frac{(m_j-m_{j-1})^2}{1+\alpha_j^2},\frac{(m_{j+1}-m_j)^2}{1+\alpha_{j+1}^2}\right\}\frac{4}{|I_j|^2} \times\\
		&\quad \times \int_{I_j}|v|^2\,rdr\\
		&\leq 4D_4\sum_{j=0}^{\infty}\int_{I_j}L_j|v|^2\,rdr\\
		&\leq 4D_4\int_0^{\infty}w_2(r)|v|^2\,rdr.
	\end{align*}
	Collecting all the constants gives
	\begin{align*}
		D_5\int_0^{\infty}\frac{|v|^2}{1+r^2}\,rdr \leq \int_0^{\infty}\left[|\partial_rv|^2 + w_2(r)|v|^2\right]\,rdr,
	\end{align*}
	where $D_5=(2\max\{D_3,\max\{D_2,D_4\}\})^{-1}$. The proof is complete.
\end{proof}

\subsection{Proofs of Theorem~\ref{thm:main_magn_est_1} and Theorem~\ref{thm:main_magn_est_2}}
\begin{proof}[Proof of Theorem~\ref{thm:main_magn_est_1}]
	We recall $h_A[u]$, given in its most general form by
	\begin{align*}
		h_A[u] := l[u] + \int_0^{\infty}\int_{\mathbb{S}^2} \left|(-i\nabla_{\omega}-rA_{\omega})u\right|^2\, d\omega\,dr.
	\end{align*}
	We know by assumption that there exists an $\omega_0 \in \mathbb{S}^2$ and $R>0$ with $B_r(R\omega_0)\neq 0$. Without loss of generality we may assume that $
	\omega_0 = (0,0,1)$, otherwise we rotate our coordinates. Since $B$ is continuous, $B_r$ is non-zero and of constant sign for $r=R$ and $0<\theta_0(R)\leq \theta 
	\leq \theta_1(R)<\pi$, so that the flux through this spherical cap is non-zero. By Theorem~\ref{thm:tech_magn_est_2},
	\begin{align*}
		0<M(R)\leq w_1(R,\theta), \quad 0<\theta_0(R)\leq \theta \leq \theta_1(R) < \pi.
	\end{align*}
	It is then again the continuity of $B$ that ensures that $0<\lambda(M(r),\theta_0(r),\theta_1(r))$ for $0<R-\eps \leq r \leq R+\eps$, for some $\eps > 0$. From this we 
	conclude that
	\begin{align*}
		0<w_2(r) = \frac{\lambda(M(r),\theta_0(r),\theta_1(r))}{r^2}, \quad 0< R-\eps < r < R+\eps.
	\end{align*}
	and the result follows from Theorem~\ref{thm:tech_magn_est_3}.
\end{proof}

\begin{proof}[Proof of Theorem~\ref{thm:main_magn_est_2}]
	We again suppose that we can measure the flux with respect to the axis $(0,0,1)$, otherwise we rotate the coordinate frame. We claim that in this setting the conditions 
	of Theorem~\ref{thm:tech_magn_est_4} are satisfied. Indeed, if $\Phi(r,\theta)$ satisfies this two-sided estimate, then
	\begin{align*}
		0<w_2(r) = \frac{\lambda(C_2,\theta_0,\theta_1)}{r^2} =: \frac{C_5}{r^2}.
	\end{align*}
	If we define 
	\begin{align*}
		I_j &= (r_0+j,r_0+(j+1)),\\
		L_j &= \frac{C_5}{(r_0+j)^2},
	\end{align*}
	then the conditions of Theorem~\ref{thm:tech_magn_est_4} are easily seen to be true.
\end{proof}

\section{Special Examples}
\subsection{Asymptotics for Weak Fields}
For a general magnetic field $B$, it is interesting to study the behaviour of the potential $w_2$ in the limit $\alpha \to 0$ when $B$ is replaced with the field $B' = \alpha B$. 
By \eqref{eq:ev_flux_comp}, $\Phi'(r,\theta) = \alpha \Phi(r,\theta)$. In the limit $\alpha \to 0$, the flux will be small enough so that the minimum over all integers will be 
achieved for $k=0$, and
\begin{align*}
	w_2(\alpha) = \alpha^2 w_2,
\end{align*}
meaning that if we turn on a magnetic field, the contribution appears quadratic in the field strength.

\subsection{Ahronov-Bohm Type Fields}
We define an Ahronov-Bohm type potential in three dimensions as follows. Let
\begin{align*}
	A(x) = \frac{\alpha}{|x|}\left(\frac{-x_2}{\sqrt{x_1^2+x_2^2}}, \frac{x_1}{\sqrt{x_1^2+x_2^2}},0\right), \quad 0<\alpha \in \mathbb{R},
\end{align*}
or in spherical coordinates
\begin{align*}
	A(r,\theta,\phi) = \frac{\alpha}{r}e_\phi.
\end{align*}
A short computation reveals
\begin{align*}
	B = \alpha \frac{\cot\theta}{r^2} e_r
\end{align*}
and $\divg B = 0$ except on $\mathbb{I}_3$, where the field is singular. This vector potential generates no continuous magnetic field, yet the potential is in $L_{\loc}^2(\R^3)$, so the quadratic form is well defined. Even though the  regularity assumptions of our theorems are not fulfilled, spectral results can still be obtained due to the simplicity of the magnetic potential. This example is interesting from the point of view that if one wants to improve the constant $1/4$ of the classical Hardy term, then the field has to be sufficiently singular near the origin, something that is in analogy to the two-dimensional case.
\begin{remark}
	Note that this is not the two-dimensional Ahronov-Bohm potential lifted to three dimensions, $A(r,\theta,\phi) = \alpha(r\sin\theta)^{-1}e_{\phi}$. We do not intend to 
	study this potential, since its components are not in $L_{\loc}^2(\R^3)$, something that causes delicate technical problems. 
\end{remark}
For this magnetic vector potential we obtain
\begin{align*}
	h_A[u] \geq \int_0^\infty \int_0^{\pi}\int_0^{2\pi} \left[|\partial_\theta u|^2 + \frac{1}{\sin^2\theta}\left|(-i \partial_\phi - \alpha\sin\theta)u\right|^2\right]\, \sin\theta d\theta d
	\phi dr
\end{align*}
Dropping the kinetic term in $\theta$ and proceeding with spectral analysis of the operator on the circle immediately implies the estimate
\begin{align*}
	h_A[u] &\geq \min_{\theta \in (0,\pi)}\min_{k \in \mathbb{Z}}\left(\frac{k}{\sin\theta} - \alpha\right)^2 \int_{\R^3} \frac{|u(x)|^2}
	{|x|^2}\,dx\\
	&= \min(\alpha,1-\alpha)^2 \int_{\R^3} \frac{|u(x)|^2}{|x|^2}\,dx
\end{align*}
for $0<\alpha<1$. When $\alpha \geq 1$, there exist a $k_0$ and a $\theta_0$ so that $k_0/\sin\theta_0 = \alpha$, so the inequality ceases to hold.

The situation can however be remedied with the help of Lemma~\ref{lem:low_ev_est}. As in the proof of Theorem~\ref{thm:tech_magn_est_2} we keep the kinetic term in $\theta$ so that
\begin{align*}
	h_A[u] &\geq \int_0^\infty \int_0^{\pi}\int_0^{2\pi} \left[|\partial_\theta u|^2 + \frac{1}{\sin^2\theta}\left|(-i \partial_\phi - \alpha\sin\theta)u\right|^2\right]\, \sin\theta d\theta d
	\phi dr\\
	&\geq \int_0^\infty \int_0^{\pi}\int_0^{2\pi} \left[|\partial_\theta u|^2 + \frac{\min_{k\in \mathbb{Z}}(k-\alpha\sin\theta)^2}{\sin^2\theta}|u|^2\right]\, \sin\theta d\theta d\phi dr.
\end{align*}
For any $\alpha>0$ we pick $\theta_0=\arcsin(1/2\alpha)/2$ and $\theta_1=\arcsin(1/2\alpha)$, so that
\begin{align*}
	0 < \alpha^2 \leq \frac{\min_{k \in \mathbb{Z}}(k-\alpha \sin\theta)^2}{\sin^2\theta},\quad 0<\theta_0\leq \theta \leq \theta_1 < \pi.
\end{align*}
and thus
\begin{align*}
	h_A[u] \geq C'(\alpha)\int_{\R^3} \frac{|u(x)|^2}{|x|^2}\,dx.
\end{align*}
This shows that the constant $1/4$ can be improved to $1/4+C'(\alpha)$ with the help of a magnetic field, $C'(\alpha)$ being an universal constant computed through \eqref{def:lambda_const}. Our methods provide unfortunately no information on the optimality of the constant.

\appendix
\section{Proof of the Technical Lemma}
\begin{proof}[Proof of Lemma~\ref{lem:poincare_type_ineq}]
	It suffices to prove the inequality in the case when $f \in C^{\infty}([0,\pi])\, \cap\, H^1((0,\pi);\sin\theta d\theta)$, since this space is dense in 
	$H^1((0,\pi);\sin\theta d\theta)$. This stems from the fact that 
	\begin{align*}
		\frac{1}{\sqrt{2}}\norm{f}_{H^1((0,\pi);d(\theta)d\theta)} \leq \norm{f}_{H^1((0,\pi);\sin\theta d\theta)} \leq \norm{f}_{H^1((0,\pi);d(\theta)d\theta)},
	\end{align*}
	where $d(\theta) = \dist(\theta,\partial(0,\pi))$, and
	\begin{align*}
		\overline{C^{\infty}([0,\pi])\, \cap\, H^1((0,\pi);d(\theta)d\theta)} = H^1((0,\pi);d(\theta)d\theta)
	\end{align*}
	was shown in \cite{Ku}.
	
	Assume first that $0<c\leq \pi/2$. Also, we start by restricting our attention to the interval $(0,c)$ and we write $f(\theta) = -\int_{\theta}^cf'(t)\,dt$, so that
	\begin{align*}
		\int_0^c |f(\theta)|^2\,\sin\theta d\theta &\leq \int_0^c \sin\theta \left(\int_\theta^c|f'(t)|\,dt\right)^2 d\theta\\
		&\leq \int_0^c \sin\theta (c-\theta) \left(\int_\theta^c|f'(t)|^2\,dt\right) d\theta
	\end{align*}
	where the last step follows from the Cauchy-Schwartz inequality. We then use that $\sin\theta \leq \sin t$, for $\theta \leq t$ by our assumption on $c$, so that
	\begin{align*}
		\int_0^c |f(\theta)|^2\,\sin\theta d\theta &\leq \int_0^c (c-\theta) \left(\int_\theta^c|f'(t)|^2\,\sin t dt\right) d\theta\\
		&\leq \int_0^c (c-\theta)\,d\theta \int_0^c|f'(t)|^2\,\sin t dt\\
		&= \frac{c^2}{2}\int_0^c|f'(\theta)|^2\,\sin\theta d\theta.
	\end{align*}
	On the interval $(c,\pi)$ we write $f(\theta) = \int_c^{\theta}f'(t)\,dt$, and analogous to the previous computation we obtain
	\begin{align*}
		\int_c^\pi |f(\theta)|^2\,\sin\theta d\theta \leq \int_c^\pi \sin\theta (c-\theta) \left(\int_c^\theta|f'(t)|^2\,dt\right) d\theta.
	\end{align*}
	By our choice of $c$, $\sin\theta \leq \frac{\sin t}{\sin c}$ for $c \leq t \leq \theta$, hence
	\begin{align*}
		\int_c^\pi |f(\theta)|^2\,\sin\theta d\theta &\leq \int_c^{\pi} \frac{(c-\theta)}{\sin c}\left(\int_{c}^{\theta}|f'(t)|^2\,\sin t dt\right)\,d\theta\\
		&\leq \frac{(\pi-c)^2}{2\sin c}\int_c^{\pi}|f'(\theta)|^2\,\sin\theta d\theta.
	\end{align*}
	We can thus conclude that for this choice of $c$ we have
	\begin{align*}
		k_1(c) = \max\left\{\frac{c^2}{2},\frac{(\pi-c)^2}{2\sin c}\right\}.
	\end{align*}
	
	The proof for $\pi/2 < c < \pi$ is very similar, and one obtains
	\begin{align*}
		k_1(c) = \max\left\{\frac{c^2}{2\sin c},\frac{(\pi-c)^2}{2}\right\},
	\end{align*}
	so that for arbitrary $c \in (0,\pi)$ the constant $k_1(c)$ computes to
	\begin{align*}
		k_1(c) = \frac{\max\{c^2,(c-\pi)^2\}}{2\sin c}.
	\end{align*}
\end{proof}

\newpage
\paragraph{Aknowledgements.} It is a pleasure to thank A. Laptev, A. Sobolev, M. Fraas, G.M. Graf, D. Lundholm and H. Kovarik for fruitful discussion. T.E. is supported by the Swedish Research Council grant Nr. FS-2009-493, F.P. by Nr. 80504801. Furthermore, financial support from the Svenska Matematiker Samfundet is gratefully acknowledged. A large part of this work was done at the Institute Mittag-Leffler in Stockholm during the program "Hamiltonians in Magnetic Fields"; both authors would like to thank the organisers and the staff.


\vspace{1cm}
\noindent
Tomas Ekholm and Fabian Portmann\\
KTH Royal Institute of Technology\\
Department of Mathematics\\
Lindstedtsv\"agen 25\\
10044 Stockholm, Sweden
\begin{tabbing}
e-mail: \=\href{mailto:fabianpo@math.kth.se}{fabianpo@math.kth.se}\\
\>\href{mailto:tomase@math.kth.se}{tomase@math.kth.se}
\end{tabbing}
\end{document}